\documentclass[12pt]{amsart}
\pdfoutput=1

\usepackage[margin=1in,marginparwidth=0.8in,marginparsep=0.1in]{geometry}
\usepackage{times}
\usepackage{soul}

\usepackage{etoolbox}

\makeatletter
\let\old@tocline\@tocline
\let\section@tocline\@tocline
\newcommand{\subsection@dotsep}{4.5}
\newcommand{\subsubsection@dotsep}{4.5}
\patchcmd{\@tocline}
  {\hfil}
  {\nobreak
     \leaders\hbox{$\m@th
        \mkern \subsection@dotsep mu\hbox{.}\mkern \subsection@dotsep mu$}\hfill
     \nobreak}{}{}
\let\subsection@tocline\@tocline
\let\@tocline\old@tocline

\patchcmd{\@tocline}
  {\hfil}
  {\nobreak
     \leaders\hbox{$\m@th
        \mkern \subsubsection@dotsep mu\hbox{.}\mkern \subsubsection@dotsep mu$}\hfill
     \nobreak}{}{}
\let\subsubsection@tocline\@tocline
\let\@tocline\old@tocline

\let\old@l@subsection\l@subsection
\let\old@l@subsubsection\l@subsubsection

\def\@tocwriteb#1#2#3{%
  \begingroup
    \@xp\def\csname #2@tocline\endcsname##1##2##3##4##5##6{%
      \ifnum##1>\c@tocdepth
      \else \sbox\z@{##5\let\indentlabel\@tochangmeasure##6}\fi}%
    \csname l@#2\endcsname{#1{\csname#2name\endcsname}{\@secnumber}{}}%
  \endgroup
  \addcontentsline{toc}{#2}%
    {\protect#1{\csname#2name\endcsname}{\@secnumber}{#3}}}%

\newlength{\@tocsectionindent}
\newlength{\@tocsubsectionindent}
\newlength{\@tocsubsubsectionindent}
\newlength{\@tocsectionnumwidth}
\newlength{\@tocsubsectionnumwidth}
\newlength{\@tocsubsubsectionnumwidth}
\newcommand{\settocsectionnumwidth}[1]{\setlength{\@tocsectionnumwidth}{#1}}
\newcommand{\settocsubsectionnumwidth}[1]{\setlength{\@tocsubsectionnumwidth}{#1}}
\newcommand{\settocsubsubsectionnumwidth}[1]{\setlength{\@tocsubsubsectionnumwidth}{#1}}
\newcommand{\settocsectionindent}[1]{\setlength{\@tocsectionindent}{#1}}
\newcommand{\settocsubsectionindent}[1]{\setlength{\@tocsubsectionindent}{#1}}
\newcommand{\settocsubsubsectionindent}[1]{\setlength{\@tocsubsubsectionindent}{#1}}

\renewcommand{\l@section}{\section@tocline{1}{\@tocsectionvskip}{\@tocsectionindent}{}{\@tocsectionformat}}%
\renewcommand{\l@subsection}{\subsection@tocline{2}{\@tocsubsectionvskip}{\@tocsubsectionindent}{}{\@tocsubsectionformat}}%
\renewcommand{\l@subsubsection}{\subsubsection@tocline{3}{\@tocsubsubsectionvskip}{\@tocsubsubsectionindent}{}{\@tocsubsubsectionformat}}%
\newcommand{\@tocsectionformat}{}
\newcommand{\@tocsubsectionformat}{}
\newcommand{\@tocsubsubsectionformat}{}
\expandafter\def\csname toc@1format\endcsname{\@tocsectionformat}
\expandafter\def\csname toc@2format\endcsname{\@tocsubsectionformat}
\expandafter\def\csname toc@3format\endcsname{\@tocsubsubsectionformat}
\newcommand{\settocsectionformat}[1]{\renewcommand{\@tocsectionformat}{#1}}
\newcommand{\settocsubsectionformat}[1]{\renewcommand{\@tocsubsectionformat}{#1}}
\newcommand{\settocsubsubsectionformat}[1]{\renewcommand{\@tocsubsubsectionformat}{#1}}
\newlength{\@tocsectionvskip}
\newcommand{\settocsectionvskip}[1]{\setlength{\@tocsectionvskip}{#1}}
\newlength{\@tocsubsectionvskip}
\newcommand{\settocsubsectionvskip}[1]{\setlength{\@tocsubsectionvskip}{#1}}
\newlength{\@tocsubsubsectionvskip}
\newcommand{\settocsubsubsectionvskip}[1]{\setlength{\@tocsubsubsectionvskip}{#1}}

\patchcmd{\tocsection}{\indentlabel}{\makebox[\@tocsectionnumwidth][l]}{}{}
\patchcmd{\tocsubsection}{\indentlabel}{\makebox[\@tocsubsectionnumwidth][l]}{}{}
\patchcmd{\tocsubsubsection}{\indentlabel}{\makebox[\@tocsubsubsectionnumwidth][l]}{}{}

\newcommand{\@sectypepnumformat}{}
\renewcommand{\contentsline}[1]{%
  \expandafter\let\expandafter\@sectypepnumformat\csname @toc#1pnumformat\endcsname%
  \csname l@#1\endcsname}
\newcommand{\@tocsectionpnumformat}{}
\newcommand{\@tocsubsectionpnumformat}{}
\newcommand{\@tocsubsubsectionpnumformat}{}
\newcommand{\setsectionpnumformat}[1]{\renewcommand{\@tocsectionpnumformat}{#1}}
\newcommand{\setsubsectionpnumformat}[1]{\renewcommand{\@tocsubsectionpnumformat}{#1}}
\newcommand{\setsubsubsectionpnumformat}[1]{\renewcommand{\@tocsubsubsectionpnumformat}{#1}}
\renewcommand{\@tocpagenum}[1]{%
  \hfill {\mdseries\@sectypepnumformat #1}}

\let\oldappendix\appendix
\renewcommand{\appendix}{%
  \leavevmode\oldappendix%
  \addtocontents{toc}{%
    \protect\settowidth{\protect\@tocsectionnumwidth}{\protect\@tocsectionformat\sectionname\space}%
    \protect\addtolength{\protect\@tocsectionnumwidth}{2em}}%
}
\makeatother

\makeatletter
\settocsectionnumwidth{2em}
\settocsubsectionnumwidth{2.5em}
\settocsubsubsectionnumwidth{3em}
\settocsectionindent{1pc}%
\settocsubsectionindent{\dimexpr\@tocsectionindent+\@tocsectionnumwidth}%
\settocsubsubsectionindent{\dimexpr\@tocsubsectionindent+\@tocsubsectionnumwidth}%
\makeatother

\settocsectionvskip{5pt}
\settocsubsectionvskip{0pt}
\settocsubsubsectionvskip{0pt}
    
\settocsectionformat{\bfseries}
\settocsubsectionformat{\mdseries}
\settocsubsubsectionformat{\mdseries}
\setsectionpnumformat{\bfseries}
\setsubsectionpnumformat{\mdseries}
\setsubsubsectionpnumformat{\mdseries}

\let\oldtableofcontents\tableofcontents
\renewcommand{\tableofcontents}{%
  \vspace*{-\linespacing}% Default gap to top of CONTENTS is \linespacing.
  \oldtableofcontents}

\usepackage{amsmath,amssymb,amsthm,amsfonts,amscd,mathrsfs}

\usepackage[all,knot,poly]{xy}
\usepackage{tikz,tikz-cd,tikz-3dplot}
\usetikzlibrary{
  arrows.meta,
  patterns,
  knots,
  calc,
  math,
  decorations,
  decorations.markings,
  shapes.misc
}
\usepackage{graphicx,overpic,float}
\usepackage{quiver}

\usepackage{xcolor}
\usepackage{soul}
\usepackage[draft]{say}

\usepackage[hidelinks]{hyperref}

\usepackage{epigraph}

\tikzset{
  anchorbase/.style={baseline={([yshift=-0.5ex]current bounding box.center)}},
  cross/.style={
    cross out,
    draw=black,
    minimum size=2*(#1-\pgflinewidth),
    inner sep=0pt,
    outer sep=0pt
  },
  cross/.default={1pt},
  partial ellipse/.style args={#1:#2:#3}{insert path={+ (#1:#3) arc (#1:#2:#3)}}
}
\tikzstyle directed=[postaction={decorate,decoration={markings,mark=at position #1 with {\arrow{>}}}}]

\definecolor{egyptianblue}{rgb}{0.06,0.2,0.65}
\definecolor{caputmortuum}{rgb}{0.35,0.15,0.13}

\newtheorem{lemma}{Lemma}

\newtheorem{corollary}[lemma]{Corollary}
\newtheorem{theorem}[lemma]{Theorem}

\theoremstyle{definition}

\theoremstyle{remark}
\newtheorem{remark}[lemma]{Remark}

\theoremstyle{expectation}

\newcommand{\Q}{\mathbb{Q}}

\newcommand{\Z}{\mathbb{Z}}

\DeclareMathSymbol{\shortminus}{\mathbin}{AMSa}{"39}

\makeatletter
\def\namedlabel#1#2{%
  \begingroup
    #2%
    \def\@currentlabel{#2}%
    \phantomsection
    \label{#1}%
  \endgroup
}
\makeatother

\title{Welschinger invariants and the Conway polynomial}

\author{Vivek Shende}
\author{Wennan Zhang}
\date{}

\begin{document}

\begin{abstract}
% It has been shown that holomorphic curves with boundaries on Lagrangian 3-manifolds (equipped with a bounding 4-chain) can be invariantly counted using the class of their boundaries in the framed HOMFLYPT skein. 

% Here we observe that, in the setting where curves must pass through at least one boundary constraint, the framed skein module is 
Welschinger showed that counts of connected holomorphic disks with Lagrangian boundary in symplectic 6-manifolds, meeting at least one boundary constraint, can be made invariant by correcting them with counts of disconnected disks weighted by certain  ``self-linking'' numbers.  
We  show  his invariant is the lowest order term in an all-genus curve count where curves are weighted by the Conway polynomials of their boundaries.
This in turn is a specialization of the skein-valued curve count, but can be defined without the 4-chain and vector field used in that setup.  
\end{abstract}

\maketitle

%\tableofcontents

\thispagestyle{empty}

The compactified moduli spaces of holomorphic maps from curves-with-boundary themselves have boundaries.  In order to extract invariants, several methods have been introduced to cancel said boundary phenomena with each other.  
Early works used antiholomorphic involutions for this purpose \cite{Welschinger-real-4, welschinger-real-6, solomon-real}.  
In the absence of such structures, Fukaya proposed to count disks in Calabi-Yau 3-folds using his algebraic theory of bounding cochains \cite{Fukaya-disks}.  Welschinger \cite{Welschinger-six} proposed a different approach: for 3-folds,  when there is one boundary  constraint, one can  count possibly disconnected disks by a certain ``self-linking number'' and recover invariance.   
Solomon and Tukachinski \cite{Solomon-Tukachinsky} have given  a variant of Fukaya's setup which is sensible without a CY3 assumption and in arbitrary dimension.
Finally, Ekholm and the first-named-author have shown that in the CY3 setting, one can invariantly count curves {\em of all genera}, each by its class in the HOMFLYPT skein of the Lagrangian boundary \cite{SOB, ghost, bare}.  
To summarize: 

\vspace{2mm}
\begin{center}
    \begin{tabular}{|c||c|c|}
\hline
      curve counting methods& genus zero & any genus \\
     \hline \hline
  dimension 3  & linking numbers \cite{Welschinger-six} & skein valued counting \cite{SOB} \\
  \hline 
  any dimension & bounding cochains \cite{Solomon-Tukachinsky} & ??? \\
  \hline
\end{tabular}
\end{center}
\vspace{2mm}

It is of evident interest to explore the mystery in the bottom right of the table.  As a first step, it is natural to try and understand how the bounding cochain and skein counting respectively specialize to the linking number counting of Welschinger.  The first task has already been carried out by Chen \cite{chen-geometric, chen-STvW}. 
Here we take up the second: explaining how Welschinger's linking-number-weighted invariants can be recovered from the skein-valued curve counting.   

Let us briefly review the skein counting setup.  Given an oriented 3-manifold $M$, the framed HOMFLYPT skein module $\mathrm{Sk}(M)$ is, by definition, generated over the ring $\Z[a^{\pm}, z^{\pm}]$ by isotopy classes of framed links in $M$, subject to the following relations imposed when a given collection of framed links looks identical outside a small ball, and as the given depiction within it: 

\begin{align}
\vcenter{\hbox{
\begin{tikzpicture}[scale=0.5]
\draw[dotted] (0,0) circle (1);
\draw[ultra thick, ->] ({sqrt(2)/2},{-sqrt(2)/2}) -- ({-sqrt(2)/2},{sqrt(2)/2});
\draw[white, line width=2.5mm] ({-sqrt(2)/2},{-sqrt(2)/2}) -- ({sqrt(2)/2},{sqrt(2)/2});
\draw[ultra thick, ->] ({-sqrt(2)/2},{-sqrt(2)/2}) -- ({sqrt(2)/2},{sqrt(2)/2});
\end{tikzpicture}
}}
\;\;-\;\;
\vcenter{\hbox{
\begin{tikzpicture}[scale=0.5]
\draw[dotted] (0,0) circle (1);
\draw[ultra thick, ->] ({-sqrt(2)/2},{-sqrt(2)/2}) -- ({sqrt(2)/2},{sqrt(2)/2});
\draw[white, line width=2.5mm] ({sqrt(2)/2},{-sqrt(2)/2}) -- ({-sqrt(2)/2},{sqrt(2)/2});
\draw[ultra thick, ->] ({sqrt(2)/2},{-sqrt(2)/2}) -- ({-sqrt(2)/2},{sqrt(2)/2});
\end{tikzpicture}
}}
\;\;&=\;\;
z\;
\vcenter{\hbox{
\begin{tikzpicture}[scale=0.5]
\draw[dotted] (0,0) circle (1);
\draw[ultra thick, <-] ({sqrt(2)/2},{sqrt(2)/2}) arc (135:225:1);
\draw[ultra thick, ->] ({-sqrt(2)/2},{-sqrt(2)/2}) arc (-45:45:1);
\end{tikzpicture}
}}
\;, \label{eq:skeinrel1}
\\
a
\vcenter{\hbox{
\begin{tikzpicture}[scale=0.5]
\draw[dotted] (0,0) circle (1);
\end{tikzpicture}
}}
\;-\;
a^{-1}
\vcenter{\hbox{
\begin{tikzpicture}[scale=0.5]
\draw[dotted] (0,0) circle (1);
\end{tikzpicture}
}}
\;\;&=\;\;
z\;\vcenter{\hbox{
\begin{tikzpicture}[scale=0.5]
\draw[dotted] (0,0) circle (1);
\draw[ultra thick, ->] (0.5,0) arc (0:370:0.5);
\end{tikzpicture}
}}
\;, \label{eq:skeinrel2}
\\
\vcenter{\hbox{
\begin{tikzpicture}[scale=0.17]
\draw[dotted] (0, 0) circle (3);
\draw [ultra thick] (1,-1) to [out=180,in=-90] (0,0);
\draw [ultra thick, ->] (0,0) -- (0,3);
\draw [ultra thick] (1,1) to [out=0,in=90] (2,0) to [out=-90,in=0] (1,-1);
\draw [white, line width=2.5mm] (0,-3) to [out=90,in=-90] (0,0) to [out=90,in=180] (1,1);
\draw [ultra thick] (0,-3) to [out=90,in=-90] (0,0) to [out=90,in=180] (1,1);
\end{tikzpicture}
}}
\;\;&=\;\;
a\;
\vcenter{\hbox{
\begin{tikzpicture}[scale=0.17]
\draw[dotted] (0, 0) circle (3);
\draw[ultra thick, <-] (0, 3) -- (0, -3);
\end{tikzpicture}
}}
\;. \label{eq:skeinrel3}
\end{align}
The `unnormalized' HOMFLYPT polynomial \cite{HOMFLY, PT} gives an equivalence \begin{eqnarray}
\mathcal{F} : \mathrm{Sk}(S^3) & \xrightarrow{\sim} & \Z[a^{\pm}, z^{\pm}]\\
\notag \emptyset & \mapsto & 1
\end{eqnarray} 
Note that in this (un)normalization, the value on the unknot is $\mathcal{F}(\bigcirc) = \frac{a-a^{-1}}{z}$.  

Fix a Calabi-Yau 3-fold $X$, Maslov class zero Lagrangian $L \subset X$, and class $d \in H_2(X, L)$.
Choose a generic almost complex structure $J$.  We fix certain `brane data': as usual, a spin structure on $L$, but in addition, a vector field $v$ on $L$ and a 4-chain $V$ such that $\partial V = 2L$ and $V \cap \mathrm{nbd}(L) \sim \pm J\cdot v$.  
After appropriate perturbation $P$, one may ensure that all (perturbed) `bare' (without components of symplectic area zero) pseudoholomorphic curves are embedded and meet $L$ in a link framed by $v$ \cite{bare}.  We write $\mathcal{M}(X, L, d)$ for the moduli space of such, possibly disconnected, curves, and $\widehat{\mathrm{Sk}}(L) := \mathrm{Sk}(L) \otimes_{\mathbb{Z}[z^{\pm}]} \mathbb{Q}[z^{-1},z]]$. 
  The skein-valued curve-count is \cite[Def. 6.2]{SOB}: 
\begin{equation} \label{skein valued curve count}
    Z_{X, L, d} \,\,\,\, = \!\!\!\! \sum_{u \in \mathcal{M}(X, L, d)} \!\!\!\! |u| \cdot z^{-\chi(u)} \cdot a^{u \cdot V} \cdot [\partial u] \in \widehat{\mathrm{Sk}}(L) 
\end{equation}
Here, $|u| \in \mathbb{Q}$ is the weight of the point in the perturbed moduli space (the fractions come because the perturbation is a branched multisection of an orbibundle); 
$\chi(u)$ is the Euler characteristic of the domain of the map $u$; by $u \cdot V$ we mean the intersection number; and $[\partial u]$ is the class in $\mathrm{Sk}(L)$ of the image under $u$ of the boundary of the domain of $u$.  

This count is independent of generic choice of $J$ and $P$ and locally constant in $v, V$ (\cite[Thm. 6.3]{SOB} if multiple covers are excluded a priori, and in \cite[Cor. 10.2]{bare} in general).  
The basic idea behind invariance is that there are two kinds of possible degenerations in 1-parameter families: a breaking along the boundary (`hyperbolic node'), and the collapse of a boundary to a point (`elliptic node').  It turns out that these degenerations, considered together with the branches of moduli space passing through the normalizations of these maps, exactly pick out the three terms of Skein Relations \eqref{eq:skeinrel1} and \eqref{eq:skeinrel2}, respectively.  The 4-chain $V$ is used to introduce `a' parameters in \eqref{eq:skeinrel2}, and the framing vector field $v$ is introduce to cancel the resulting boundary term $u\cdot \partial V$ via Skein Relation \eqref{eq:skeinrel3}.

\vspace{2mm}

We recall Welschinger's setup.  
Fix a 3-manifold $M$ and ring $A$.  Recall that  an oriented link $K = K_1 \cup K_2 \subset M$ with  $[K_1], [K_2] = 0 \in H_1(M, A)$, has a linking number $\mathrm{Link}(K_1, K_2) \in A$, symmetric in the factors, and computed by intersecting $K_1$ with a chain bounding $K_2$.  In particular, two-component links in rational homology spheres have linking numbers valued in $\mathbb{Q}$.  

Given a link $K = K_1 \cup \cdots \cup K_n $ with all $[K_i] = 0 \in H_1(M, A)$, form the complete graph  whose vertices are the components of $K$, and label each edge by the linking number of the corresponding components; we denote this labelled graph $\Gamma_K$.  Then Welschinger defines $\mathrm{Link}(K) \in A$ by evaluating the Kirchhoff polynomial of the graph $\Gamma_A$ at the specified labels, i.e. as the sum, over spanning trees of $\Gamma_K$, of the products, over edges of the tree, of the labels of the edges.  By definition, if $K$ has one component, $\mathrm{Link}(K) = 1$.  

Fix a symplectic 6-fold $X$ and Lagrangian $L \subset X$.  Assume the map $H_2(X, L; A) \to H_1(L; A)$ is the zero map, so that for a smooth embedded holomorphic curve $\gamma$ ending on $L$, the quantity $\mathrm{Link}(\partial \gamma)$ is well defined.  Fix a homology class $d \in H_2(X, L)$.  Fix homology classes  $\xi_1, \ldots, \xi_k$ on $X$ and $\lambda_1, \ldots, \lambda_m$ on $L$, so that the moduli of disks with $k$ interior and $m$ boundary marked points, mapping to  generic cycle representatives of $\xi_i$ and $\lambda_j$, is zero dimensional.  {\em Assume in addition that the number of boundary constraints is $m \ge 1$,} and that there are no Maslov zero disks to avoid multiple covers.\footnote{The hypothesis on no Maslov zero disks can be removed using appropriate perturbations, e.g. via \cite{bare}.}  We will write $\mathbf{d}$ for the data of $d$ along with the $\xi_i$ and $\lambda_j$.  
For generic $J$, such disks are embedded, and the moduli space is transversely cut out.  We write $\mathcal{M}^{disk}(X, L, \mathbf{d})$ for the moduli space of such, possibly disconnected, disks.  Welschinger showed that the following quantity is independent of the choice of generic $J$, and the cycle representatives for the $\xi_i$ and $\lambda_j$: 
\begin{equation}
    W_{X, L, \mathbf{d}} \,\,\,\, := \!\!\!\!\!\!\!\! \sum_{u \in \mathcal{M}^{disk}(X, L, \mathbf{d})} \!\!\!\!\!\!\!\!  |u| \cdot \mathrm{Link}(\partial u)
\end{equation}
In Welschinger's setting, there are (by hypothesis) no multiple covers, so $|u| = \pm 1$.

\vspace{2mm}

In fact, when $M = S^3$, Welschinger's linking invariant has appeared previously in knot theory.  Recall that Conway's normalized  Alexander polynomial \cite{Conway-polynomial}  is the unique function 
$\nabla$ from (not framed) links in $S^3$ to $\Z[z]$ which satisfies Skein Relation \eqref{eq:skeinrel1} and takes $\nabla(\bigcirc) = 1$.  As is apparent from the skein relations, it is a specialization of the HOMFLYPT polynomial: 
\begin{equation}
    \nabla(K) = \frac{\mathcal{F}(K)}{\mathcal{F}(\bigcirc)} \bigg|_{a=1}  
\end{equation}
One can  see  from the skein relations: if $K$ is a link with $n$ components, then $\nabla(K) \in z^{n-1} \Z[z]$, and  if $K$ is a separable link (has components which map be enclosed by disjoint balls), then $\nabla(K) = 0$. 

\begin{lemma} \label{first Conway coeff classical}
    \cite{Hartley-potential, Hoste-first}
    The coefficient of $z^{n-1}$ in $\nabla(K)$ is $\mathrm{Link}(K)$.
\end{lemma}
\begin{proof}
    The original arguments proceed by studying the Seifert matrix.  One can also give a proof using just the skein relation and facts Welschinger established directly for $\mathrm{Link}(K)$.  If $K$ has one component, it is well known (and immediate from the skein relations) that $\nabla(K)$ has constant coefficient $1$.  We induct on the number of link components.  
    Given any nontrivial link, we may unlink a component by only changing crossings where the two branches are on different components of the link; 
    Welschinger's \cite[Lemma 2.3]{Welschinger-six} shows that $\mathrm{Link}(K)$ transforms compatibly with Skein Relation \eqref{eq:skeinrel1}. Once the component is unlinked, both $\nabla(K)$ and $\mathrm{Link}(K)$ vanish. 
\end{proof}

Boyer and Lines have generalized the Conway polynomial to an invariant $\nabla_{BL}$  for links in rational homology spheres \cite{boyer-lines}.  Their invariant  takes values in $\Q[s^{\pm 1/m}, (s^{1/n} - s^{-1/n})^{-1}]_{m,n \in \Z}$,  and satisfies Skein Relation \eqref{eq:skeinrel1} where $z = s - s^{-1}$; we later also set $s=e^x$. 
The comparison to the Alexander polynomial given in \cite[Property (I)]{boyer-lines} implies that 
$z \nabla_{BL}(K) \in \Q[z^{-1},z]] \subset \Q[x^{-1}, x]]$; the symmetry property of \cite[Property (V)]{boyer-lines} implies that only either even or odd powers of $z$ or $x$ can appear.   Comparing our normalization convention to theirs, for the 3-sphere, $\nabla = z \nabla_{BL}$; and we preserve this notationally for other rational homology spheres.  

For an oriented 3-manifold $M$, we write $\mathrm{Co}(M)$ for its Conway skein: the $\Z[z]$-module generated by (not framed) links, subject to Skein Relation \eqref{eq:skeinrel1} and the relation $\bigcirc = 1$, then, for any rational homology 3-sphere, the function $\nabla = z \nabla_{BL}$ gives a map
$$\nabla: \mathrm{Co}(M) \to \Q[[z]]$$
which specializes to the Conway polynomial when $M = S^3$.

\begin{lemma}\label{lem: also for QHS}
    % Let $M$ be a closed oriented rational homology sphere and
    % $L=K_1\cup\cdots\cup K_n\subset M$ an oriented link.
    % Write
    % \[
    %     \ell_{ij}=\operatorname{lk}_M(K_i,K_j)\in\Q
    %     \qquad (i\neq j).
    % \]
    % Let $\nabla_{BL}$ denote the Boyer--Lines potential function,
    % and define
    % \[
    %     \nabla(L;z)
    %     :=z\,\nabla_{BL}(L;s(z),\ldots,s(z)),
    %     \qquad
    %     s(z)=\frac{z+\sqrt{z^2+4}}{2},
    % \]
    % so that $s(z)-s(z)^{-1}=z$ and $s(0)=1$.
    % Then
    % \[
    %     \nabla(L;z)\in z^{n-1}\Q[[z^2]]
    % \]
    % and
    % \[
    %     [z^{n-1}]\nabla(L;z)
    %     =
    %     \sum_{T\in\mathcal T_n}
    %     \prod_{\{i,j\}\in E(T)}\ell_{ij},
    % \]
    % where $\mathcal T_n$ is the set of spanning trees on
    % $\{1,\ldots,n\}$.
    % For $n=1$, the sum is understood to be $1$.
    % Thus this coefficient is $\mathrm{Link}(L)$, with the
    % spanning-tree convention used by Welschinger.
    % It need not be nonzero.
    For an $n$-component link $K = K_1 \cup \cdots K_n$  in a rational homology sphere, we have $\nabla(K) \in z^{n-1} \Q[[z]]$, and the coefficient of $z^{n-1}$ is $\mathrm{Link}(K)$.
\end{lemma}
\begin{remark}
    The skein reduction argument of Lemma \ref{first Conway coeff classical} does not generalize, since one can't always reduce to an unlinked case by crossing changes.   
   If we knew  the Alexander polynomial can be computed using a rational Seifert surface, we could likely adapt \cite{Hoste-first, Hartley-potential}, but we did not find a proof of this assertion in the literature (though see \cite{cha-ko}).  An argument of Buryak \cite{buryak2011first} does generalize; we copy it here to illustrate that it only uses general properties of $\nabla_{BL}$ established in \cite{boyer-lines}. 
\end{remark}
\begin{proof}
    The case where $K$ is a knot follows from \cite[Properties (I), (II)]{boyer-lines}.  
    % First suppose that $L=K$ is a knot. By properties (I) and
    % (II) of \cite{boyer-lines},
    % \[
    %     \nabla_{BL}(K;s)
    %     =
    %     \frac{\Delta_K(s^2)}{s^d-s^{-d}},
    %     \qquad
    %     \Delta_K(1)=d,
    % \]
    % where
    % \[
    %     d=
    %     \frac{|\operatorname{Tor}H_1(M\setminus K;\Z)|}
    %          {|H_1(M;\Z)|}.
    % \]
    % Consequently,
    % \[
    %     \lim_{s\to1}
    %     (s-s^{-1})\nabla_{BL}(K;s)=1.
    % \]
    % The singularity is removable after multiplication by
    % $s-s^{-1}$. The symmetry property of $\nabla_{BL}$ then
    % gives
    % \[
    %     \nabla(K;z)\in 1+z^2\Q[[z^2]].
    % \]
    % This proves the assertion for $n=1$.
    We will use the multi-variable  (one for each component) Conway invariant.  By \cite[Property (I)]{boyer-lines} when $n \ge 2$, this satisfies   
    $\nabla_{BL}(K;e^{x_1},\ldots,e^{x_n})
   \in\Q[[x_1,\ldots,x_n]]$.
    The Torres restriction formula \cite[Property (IV)]{boyer-lines} asserts: 
    \begin{equation} \label{torres restriction}
    \nabla_{BL}(K;e^{x_1},\ldots,e^{x_{n-1}}, e^{0})
        =
        2\sinh\left(\sum_{i=1}^{n-1}\ell_{in}x_i\right)
        \nabla_{BL}(K\setminus K_n;e^{x_1},\ldots,e^{x_{n-1}}).
    \end{equation}
    The corresponding formula holds after deleting any other component as well. 
    
    Consider also the following multivariable analogue of $\mathrm{Link}(K)$, which is a homogenous polynomial of degree $n-2$: 
    \begin{equation}
    \label{multivariable link}
        \mathrm{Link}(K; x_1,\ldots,x_n)
        :=
        \sum_{T\in\mathcal T_n}
        \left(
            \prod_{\{i,j\}\in E(T)}\ell_{ij}
        \right)
        \prod_{i=1}^n x_i^{\deg_T(i)-1}.
    \end{equation}
    Here, $\mathcal{T}_n$ is the set of spanning trees of the complete graph, $\ell_{ij}$ is the linking number of $K_i$ and $K_j$, and $\deg_T(i)$ is the number of edges incident to
    vertex $i$ in $T$. 
    % Since $n\ge2$, all these exponents
    % are nonnegative. Moreover, $P_L$ is homogeneous of
    % degree $n-2$, because
    % \[
    %     \sum_{i=1}^n(\deg_T(i)-1)
    %     =2(n-1)-n=n-2.
    % \]
    It satisfies the formula:
    \begin{equation} \label{link restriction}
        \mathrm{Link}(K; x_1,\ldots,x_{n-1},0)
        =
        \left(\sum_{i=1}^{n-1}\ell_{in}x_i\right)
        \mathrm{Link}(K\setminus K_n; x_1,\ldots,x_{n-1}).
    \end{equation}
    Indeed, setting $x_n=0$ retains exactly those trees
    in which vertex $n$ is a leaf. Deleting this leaf
    gives a tree on the remaining vertices; attaching
    it to vertex $i$ contributes the factor
    $\ell_{in}x_i$.
    The corresponding formula holds upon deleting any
    component.

    We will prove:
    \begin{equation} \label{to prove}
    \nabla_{BL}(K;e^{x_1},\ldots,e^{x_n})
   \equiv 2^{n-2} \mathrm{Link} (K;x_1,\ldots,x_n) \pmod{(x_1, \ldots, x_n)^{n-1}} 
   \end{equation}
    Since $z = 2x + O(x^3)$, upon setting $x_1 = \ldots = x_n = x$, we recover the statement of the theorem. 
   
    For $n = 2$, Equation \eqref{to prove}  follows from \cite[Property (II)]{boyer-lines}.
    Now suppose $n\ge3$, and we have established \eqref{to prove}  for links with
    $n-1$ components.  Let $P(x_1, \ldots, x_n)$ be the degree $\le n-2$ part of $\nabla_{BL}(K;e^{x_1},\ldots,e^{x_n}) - 2^{n-2} \mathrm{Link} (K;x_1,\ldots,x_n)$.  
    Comparing  Equations \eqref{torres restriction} and \eqref{link restriction}, we see that $P$ vanishes
     upon setting any of the $x_i$ to zero, so is divisible by $x_1 \ldots x_n$.  Since $P$ has degree $\le n-2$, we conclude $P = 0$, and hence have established \eqref{to prove}.  
\end{proof}

Thus we may rewrite Welschinger's invariant as follows: 

\begin{corollary} \label{cor: Welschinger rewritten}
    When $L$ is a rational homology sphere, 
    $$
    W_{X, L, \mathbf{d}} = \mbox{Coefficient of } z^{-1} \mbox{ in } \left[ \sum_{u \in \mathcal{M}^{disk}(X, L, \mathbf{d})} \!\!\!\!\!\!\!\!  |u| \cdot z^{-\chi(u)} \cdot \nabla(\partial u)\right]$$
\end{corollary}
\begin{proof}
    When $\partial u$ has $n$ components, then since the domain of $u$ is a union of disks, $\chi(u) = n$.  So by Lemma \ref{lem: also for QHS}, the $z^{-1}$ term of $z^{-\chi(u)} \nabla(\partial u)$ is $\mathrm{Link}(\partial u)$.
\end{proof}

It is now clear that in order to relate the skein-valued curve count \eqref{skein valued curve count} to Welschinger's invariant (at least for $L$ a rational homology sphere), we should like to set $a=1$, take the Conway polynomial, and take the coefficient of $z^{-1}$.

There is an apparent difficulty: for any knot or link the {\em unnormalized} HOMFLYPT invariant has $\mathcal{F}(K)|_{a=1} = 0$; correspondingly, $\nabla$ does not factor through $\mathrm{Sk}(M)$. 
From the point of view of skein theory, it is obvious how to fix this: take instead the variant $\mathrm{Sk}'(M)$ which is  generated by framed links {\em with at least one component} and where relation \eqref{eq:skeinrel2} is only imposed when it makes sense, i.e. we  only may evaluate an unlinked unknot component to $(a-a^{-1})/z$ if there is at least one other component of the link.  There is evidently a map: 
\begin{eqnarray} \label{conway specialization}
    \mathrm{Sk}'(M)|_{a=1} & \xrightarrow{\sim} &
    \mathrm{Co}(M) \otimes_{\Z[z]} \Z[z^{\pm}] \\
    \notag \bigcirc & \mapsto & 1. 
\end{eqnarray}  
We may  compose this with the Conway polynomial $\nabla$ when $M$ is a rational homology 3-sphere.

\vspace{2mm}
In the original setting of \cite{SOB, bare}, the skein-counting invariant cannot generally be made to take values in $\mathrm{Sk'}(L)$, because e.g. a curve with a unique boundary component may see said component contract in a 1-parameter family.   
But in the presence of at least one boundary constraint: 

\begin{theorem} \label{thm: skein counting with boundary constraint}
    Let $X$ be a symplectic 3-fold and $L\subset X$ a spin Lagrangian; choose an almost complex structure $J$ and brane data $V,v$ as usual. Fix a class $d \in H_2(X, L)$, and homology classes $\xi_1, \ldots, \xi_k$ on $X$ and $\lambda_1, \ldots, \lambda_m$ on $L$, with $m \ge 1$,  and generic chain representatives for them.   
    
   Take perturbations as in \cite{bare};  denote the moduli of bare maps from disconnected curves with marked points meeting the constraints as  
    $\mathcal{M}(X, L, \mathbf{d})$.  
    When the expected dimension of this space is zero, such curves are embedded and the moduli space is transversely cut out; define: 
    \begin{equation} \label{bmp skein valued curve count}
    Z_{X, L, \mathbf{d}} \,\,\,\, := \!\!\!\! \sum_{u \in \mathcal{M}(X, L, \mathbf{d})} \!\!\!\! |u| \cdot z^{-\chi(u)} \cdot a^{u \cdot V} \cdot [\partial u] \in \widehat{\mathrm{Sk}'}(L) 
\end{equation}
    This quantity is independent of generic choice of $J$ and perturbation and representatives for $\xi_i, \lambda_j$, and locally constant in $v, V$. 
\end{theorem}
\begin{proof}
    One runs the arguments in \cite{SOB, bare}, plus standard arguments similar to those in \cite[Sec. 6.3]{bare} to guarantee that the constraints cut transversely the moduli of bare curves. 
    The sole new point is that that boundary components carrying marked points will only contract {\em in codimension two} (whereas boundary components without marked points contract in codimension one), hence do not contribute any new degenerations in the zero- and one- parameter families needed to define and prove invariance for the skein-valued counting.

    In particular, a boundary component carrying a boundary marked point, which exists by assumption, can never degenerate in a generic 1-parameter family,  so we remain throughout in $\widehat{\mathrm{Sk}'}$.
\end{proof}

We can now specialize nontrivially at $a=1$ to get an invariant valued in $\widehat{\mathrm{Co}}(L) := \mathrm{Co}(L) \otimes_{\Z[z]} \Q[z^{-1}, z]]$.  But, there remains the following problem: in $Z_{X, L, d}$, we cannot separate closed and open contributions (see \cite[Sec. 9.2]{SOB} for discussion).  In particular, contributions from curves with closed genus zero components can contribute with $z^{-k}$ for $k > 1$ (though $k$ is bounded for fixed $d$), and e.g. a curve with one closed genus zero component and one genus one component with one boundary will contribute a $z^{-1}$ term.  

However, if we count in the Conway skein from the beginning, we may define a purely open count, and moreover do not need the 4-chain $V$ and vector field $v$:

\begin{theorem} \label{thm: open Conway skein counting}
    In the setting of Theorem \ref{thm: skein counting with boundary constraint}, but without fixing $V, v$, 
     Let 
    $\mathcal{M}^{open}(X, L, \mathbf{d})$ be the moduli of bare maps from potentially disconnected domains of any  genus, all of whose components have boundary.  Then the following is well defined and independent of $J$ and perturbation and representatives for the $\xi_i, \lambda_j$: 

    \begin{equation} \label{open Conway count}
    Y_{X, L, \mathbf{d}} \,\,\,\, = \!\!\!\! \sum_{u \in \mathcal{M}^{open}(X, L, \mathbf{d})} \!\!\!\!\!\!\!\! |u| \cdot z^{-\chi(u)} \cdot [\partial u] \in \widehat{Co}(L) 
\end{equation}
\end{theorem}
\begin{proof}
    In the proof of invariance of skein-valued curve counting in \cite[Thm. 6.3]{SOB}, the 4-chain $V$ and Skein Relation \eqref{eq:skeinrel2} are needed to deal with the boundary of moduli arising from an elliptic node degeneration; the framing vector field $v$ and Skein Relation \eqref{eq:skeinrel3} are then introduced to cancel a  boundary term newly introduced by the 4-chain.  However, if we count in the Conway skein, then whenever an elliptic node degeneration happens, the collapsing boundary is already an unlinked unknot, hence the contribution of said curve was already zero.  (This is essentially how Welschinger dealt with such degenerations \cite[Sec. 3.3.2]{Welschinger-six}.) 
\end{proof}

\begin{corollary}
    When $L$ is a rational homology sphere, 
    $W_{X, L, \mathbf{d}}$ is the coefficient of $z^{-1}$ in $\nabla(Y_{X, L, \mathbf{d}})$. 
\end{corollary}
\begin{proof}
    Any term in $z^{-\chi(u)} \nabla(\partial u)$ has degree $\#\mathrm{components}(\partial u) - \chi( u) - 1 \ge -1$, with equality only if all components are disks.  The result  follows by comparing Theorem \ref{thm: open Conway skein counting} with Corollary \ref{cor: Welschinger rewritten}.
\end{proof}

\begin{remark}
    The existence of $Y_{X, L, \mathbf{d}}$ is also interesting in the Calabi-Yau case.  Here we may retain zero dimensional moduli by taking e.g. a single boundary constraint through a 2-cycle on $L$.  The `divisor relation' as usual will identify this with some multiple of the count without marked points, a priori valued in $\mathrm{Sk}$ rather than $\mathrm{Sk}'$.  Note however that if the boundary meets said nontrivial 2-cycle homologically nontrivially, then it must be nonzero in the skein, so this component of the invariant could always have been lifted to $\mathrm{Sk}'$ and specialized at $a=1$.  Indeed, this has been done on several occasions, with various ad-hoc reasoning 
    \cite[Sec. 7.2]{ekholm-longhi-park-shende} \cite[Sec. 4.1]{hu-shende}. 
\end{remark}

\begin{remark}
    The skein-valued curve counting is consistent with Witten's view  \cite{witten-chern-string} that the topological A-model couples with the Chern-Simons theory on the Lagrangian branes; in this correspondence, one should set 
    $a=e^{Ng_s/2}$ and $z = e^{g_s/2} - e^{-g_s/2}$ if the Lagrangian carries a ``stack of $N$ branes''.  Thus the fact that some part of the open string theory survives upon setting $a=1$, i.e. $N =0$, seems mysterious from a physical point of view. 
\end{remark}

Finally, let us sum over possible $\mathbf{d} = (\{\xi_1, \ldots, \xi_k\}, \{\lambda_1, \ldots, \lambda_m\}, d)$ where the $\xi_i \in H_*(X)$
and $\lambda_j \in H_*(L)$ and $d \in H_2(X, L)$; we allow $k =0 $ but require $m \ge 1$.  We write $\mathbf{Q}^d$ for the corresponding element of the group ring; we write $\boldsymbol{\xi} = \frac{1}{k!} \xi_1 \cdots \xi_k$ and $\boldsymbol{\lambda} = \frac{1}{m!} \lambda_1 \cdots \lambda_m$, where the product is in the polynomial ring on (infinitely many) variables corresponding to elements of $H_*(X)$ and $H_*(L)$ respectively. 
We write $$W_{X, L} := \sum_{\mathbf{d}} W_{X, L, \mathbf{d}} \cdot  \mathbf{Q}^d \cdot \boldsymbol{\xi} \cdot \boldsymbol{\lambda}$$ 
and similarly for $Y, Z$.  

We consider the closed bare curve count, as defined in \cite[Sec. 9.2]{bare}, but now with marked point constraints, and denote it 
$Z_{X, \mathbf{d}}$ where now $\mathbf{d} = (\{\xi_1, \ldots, \xi_k\}, d)$.  We assemble these into $Z_X = \sum_d Z_{X, d} \cdot \mathbf{Q}^d \cdot \boldsymbol{\xi}$.  
Essentially by definition, 
\begin{equation}
    Y_{X, L} \cdot Z_X = Z_{X, L}|_{\widehat{Co}(L)}.
\end{equation}
Now assuming $L$ is a rational homology sphere, and collecting our results, we have: 
\begin{equation}
    W_{X, L} = \mbox{Coefficient of } z^{-1} \mbox{ in } \nabla(Y_{X, L}) = \mbox{Coefficient of } z^{-1} \mbox{ in }  \frac{\nabla(Z_{X,L})}{Z_X}.
\end{equation}
Thus  the Welschinger invariant is recovered from the skein-valued curve count. 

\vspace{2mm}
\noindent {\bf{Acknowledgements.}} 
We are supported by Villum Fonden Villum Investigator 37814. 

\bibliographystyle{plain}
\bibliography{ref}
\end{document}